\documentclass[11pt]{amsart}

\usepackage{amsmath,amsthm}
\usepackage{amsfonts}
\usepackage{verbatim}
\usepackage{amssymb}
\usepackage{txfonts}
\usepackage{amscd}
\usepackage{bbm}
\usepackage{hyperref}
\usepackage{mathrsfs}

\usepackage[nohug]{diagrams}

\usepackage{dsfont}

\newcounter{aux1}
\newcounter{aux2}

\newtheorem{theorem}{Theorem}

\newtheorem{lemma}[theorem]{Lemma}
\newtheorem{proposition}[theorem]{Proposition}
\newtheorem{definition}[theorem]{Definition}
\newtheorem{question}[theorem]{Question}

\newcommand{\supp}{\operatorname{supp}}

\newcommand{\Aff}{\mathscr{A}}

\newcommand{\RR}{\mathbb{R}}

\newcommand{\C}{\ell_{\infty}(G)}

\newcommand{\EL}{\mathcal{L}}

\title[Every group is weakly exact]{Every finitely generated group is weakly exact}

\author{Ronald G. Douglas}

\address{Department of Mathematics, Texas A\&M University, College Station, TX 77843-3368}
\email{rdouglas@math.tamu.edu}

\author{Piotr W. Nowak}
\address{Mathematical Sciences Research Institute, 17 Gauss Way, Berkeley, CA, USA}
\address{Institute of Mathematics of the Polish Academy of Sciences, \'{S}niadeckich 8, 00-956 Warszawa, Poland}
\address{Institute of Mathematics, University of Warsaw, Banacha 2, 02-097 Warszawa, Poland}
\email{pnowak@mimuw.edu.pl}

\keywords{exact group; bounded cohomology; Hochschild cohomology; invariant expectation}

\thanks{The second author was supported by NSF grant DMS-0900874}

\begin{document}

\begin{abstract}
We show that every finitely generated group admits weak analogues of an 
invariant expectation, whose existence characterizes exact groups.
This fact has a number of applications. We show that Hopf $G$-modules are relatively injective,
which implies that bounded cohomology groups with coefficients in all Hopf $G$-modules
vanish in all positive degrees. 
We also prove a general fixed point theorem for actions of finitely generated groups on 
$\ell_{\infty}$-type spaces. Finally, we define the notion of weak exactness for 
certain Banach algebras. 
\end{abstract}

\maketitle

In our previous work \cite{douglas-nowak}, we  studied exact groups and their
bounded cohomology.
We  also introduced the notion of Hopf $G$-modules, a class of bounded Banach
$G$-modules which are additionally equipped with a natural representation of the 
algebra $\ell_{\infty}(G)$. This work initiated the consideration of $G$-modules with an additional 
representation of a $G$-$C^*$-algebra as coefficients for the bounded cohomology. The techniques
of \cite{douglas-nowak} provided some of the key new ingredients of the
characterization of amenable actions, and in particular of exact groups, 
in terms of bounded cohomology  \cite{brodzki-niblo-nowak-wright}, \cite{monod}.

These recent results allow one to view various amenability-like properties via bounded cohomology 
in a unified manner.
The strength of these amenability-like properties corresponds precisely to the extent of the class 
of bounded $G$-modules for which the bounded cohomology vanishes.  In Johnson's classic 
theorem \cite{johnson-memoir}
characterizing amenability, this class consists of all dual modules. 
Topological amenability of an action of a group
$G$ on a compact space $X$ is detected by a subclass, the class of duals of 
$\ell_1$-geometric $G$-modules
which are additionally equipped with a compatible representation of $C(X)$ 
(see \cite{brodzki-niblo-nowak-wright}).
In this note  we are considering the class of dual Hopf $G$-modules introduced in 
\cite{douglas-nowak}. These modules 
correspond to $X=\beta G$, the Stone-\v{C}ech compactification of $G$, and certain particular representations of  $C(\beta G)\simeq\ell_{\infty}(G)$
and constitute a subclass of the previously discussed classes of test modules.

One can similarly compare various notions of amenability using averaging operators.
Amenable groups are precisely the groups for which there exists an invariant mean, 
a positive operator $\ell_{\infty}(G)\to \RR$ which is invariant under the group action.
Exact groups are characterized by the existence of an invariant expectation \cite{douglas-nowak},
a map $M:\EL(\ell_u(G),\ell_{\infty}(G))\to \ell_{\infty}(G)$, where $\EL(X,Y)$ is the space of 
bounded linear
operators from $X$ to $Y$ and $\ell_u(G)$ is the uniform convolution algebra of $G$, see 
\cite{douglas-nowak}.
The invariant expectation is required to  commute with the actions of  $G$.

The invariant expectation was the main tool used in the vanishing theorem for bounded
cohomology in \cite{douglas-nowak}. 
It also gives a convenient way to weaken 
or strengthen exactness by enlarging or reducing the space on which such an expectation is defined.
This led us to consider a condition which we initially  called \emph{weak exactness} and which
was expressed in terms of the existence of a weaker notion of an invariant expectation, defined on a space smaller than
the one needed for exactness. As it turns out,  this condition is rather mild.

\begin{theorem}
A weak invariant expectation (with coefficients in any dual module) 
exists on every finitely generated group.
\end{theorem}

Despite such generality weak invariant expectations turn out to be very
useful. We present here three applications.

First we apply the weak invariant
expectation to show that weak-* closed Hopf $G$-modules are relatively injective bounded
Banach $G$-modules (Theorem \ref{theorem : relative injectivity of Hopf G-modules}).
In particular, this implies that 
bounded cohomology groups with coefficients in weak-* closed Hopf $G$-modules vanish 
(Theorem \ref{theorem : vanishing of cohomology}) in all positive degrees.
Hopf $G$-modules are Banach subspaces of $\ell_{\infty}(G,X^*)$, where $X$ is a bounded $G$-module,
which are closed with respect to both the natural action of $G$ and the multiplicative
action of $\ell_{\infty}(G)$ and are additionally closed in the weak-* topology. 
 In \cite{douglas-nowak} we conjectured that vanishing of bounded cohomology with coefficients in
 weak-* closed Hopf $G$-modules characterizes exactness. The vanishing theorem 
 established here, somewhat surprisingly, 
 disproves this conjecture.

The second application concerns fixed points for group actions. A classic result of 
M. M. Day \cite{day} is a a characterization of amenability via a fixed point property.
Motivated by this fact  we prove a fixed point theorem for actions of discrete groups on certain compact
subsets of spaces of the $\ell_{\infty}(G,X)$ type, equipped with a weak-type topology
(Theorem \ref{theorem : fixed point}). This topology,
which we call the ultra-weak topology, is induced by $\ell_{\infty}(G,X^*)$ viewed as maps into 
$\ell_{\infty}(G)$ equipped with its weak-* topology. This fixed point theorem can be viewed as a weak
analogue of Day's theorem, which holds for all finitely generated groups.

Finally, in the last section we use the above results to define a 
notion of weak exactness for some Banach algebras (Definition \ref{definition : exact Banach algebra}). 
For $C^*$-algebras the notion of exactness
is well-studied, see \cite{brown-ozawa}, and it would be interesting to try to extend such results
to the setting of 
Banach algebras.  One can compare this with the case of a $C^*$-algebra $A$, for which amenability of $A$ 
as a Banach algebra  is equivalent to nuclearity.

We are grateful to the referee for suggesting many valuable improvements to the first version of this paper.
\setcounter{tocdepth}{1}
\tableofcontents

\section{Modules and topologies}
\subsection{Actions}
Let $G$ be a finitely generated group.
A bounded Banach $G$-module  is a Banach space $X$ with a representation of $G$ on $X$, $g\mapsto \pi_g$,  where each $\pi_g$ is a bounded linear operator on $X$,
satisfying $\sup_{g\in G}\Vert \pi_g\Vert<\infty$. Then the dual, $X^*$, is also a bounded Banach
$G$-module with the representation $\overline{\pi}_g=\pi_{g^{-1}}^*$.

In general we denote the action of $G$ on $X$
by $gx$.
Given a bounded Banach $G$-module $X$, we consider the action of $G$ on $\ell_{\infty}(G,X)$
defined by 
$$(g* f)_h=g(f_{g^{-1}h}),$$
for $g,h\in G$ and $f\in \ell_{\infty}(G,X)$.
Then the induced action on $\ell_{\infty}(G,\ell_{\infty}(G,X))$ will be denoted $g\star f$ for 
$f\in \ell_{\infty}(G,\ell_{\infty}(G,X))$ and $g\in G$.

\subsection{Pairings}
Let $X$ be a Banach space.
Denote by $1_G$ the identity in $\ell_{\infty}(G,X)$
and by $\mathds{1}_G$ the identity in 
$\ell_{\infty}(G\times G)$.
Given a function $\xi\in \ell_{\infty}(G\times G)$ we will view it as  $\xi\in \ell_{\infty}(G,\ell_{\infty}(G))$
by taking $g\mapsto \xi(g, \cdot)=\xi_g\in \ell_{\infty}(G)$. We then say that $\xi$ is finitely  supported
if there exists a finite set $F\subseteq G$ such that $\xi_g=0$ whenever $g\in G\setminus F$.  That is,
as a function on $G\times G$, $\xi$ is finitely supported in the first variable.

A finitely supported function $\xi\in\ell_{\infty}(G,\ell_{\infty}(G))$  induces a bounded linear operator
$$\langle \xi, \cdot\rangle_{\C}:\ell_{\infty}(G,\ell_{\infty}(G,X))\to\ell_{\infty}(G,X)$$
by the formula
$$\langle \xi,f\rangle_{\C}=\sum_{g\in G}\xi_g f_g.$$

Define the action of $\ell_{\infty}(G)$ on $\ell_{\infty}(G,X)$ by
multiplication:
$$(a\bullet f)_g=a_g f_g,$$
and the action of $\ell_{\infty}(G)$ on $\ell_{\infty}(G,\ell_{\infty}(G,X))$
$$(a f)_g=a\bullet \xi_g.$$
Then the operator $\langle\xi, \cdot\rangle_{\C}$ is $\ell_{\infty}(G)$-linear, in the sense that
$$\langle\xi, a f\rangle_{\C}=\langle a\xi, f\rangle_{\C}=a\langle\xi,f\rangle_{\C}.$$
For each $g\in G$ we define the element $\delta_g\in \ell_{\infty}(G,\ell_{\infty}(G))$ by setting
$$(\delta_g)_h=\left\lbrace\begin{array}{ll}
1_G&\text{ if } g=h\\
0&\text{ otherwise}.
\end{array}\right.$$
Thus $\mathds{1}_G=\sum_{g\in G}\delta_g$.

\subsection{The  weak-* operator topology on $\EL(X,M)$}

We will denote weak-* limits by $w^*-\lim$.
Let $X$ be a Banach space and $M$ be a dual space.
Consider the space $\EL(X,M)$ of bounded linear maps from $X$ to $M$, with its natural
operator norm, which we denote by $\Vert\cdot\Vert_{\EL}$.
Every element $\xi\in X$ defines a map $\hat{\xi}:\EL(X,M)\to M$ by the formula
$$\hat{\xi}(T)=T(\xi)$$
for every $T\in \EL(X,M)$. This defines a natural embedding
$$i:X \to \EL\left(\EL(X,M),M \right).$$ 

We denote the natural norm on 
$\EL\left(\EL(X,M),M \right)$ by $\Vert \cdot \Vert_{\EL\EL}$.
We have $\Vert \hat{\xi}\Vert_{\EL\EL}= \Vert \xi\Vert_{X}$ for every $\xi\in X$.
Let $\widehat{B}_{X}\subseteq \EL\left(\EL(X,M),M \right)$ denote 
the image of the unit ball  $B_{X}$ of $X$ under the inclusion $i$.
\begin{definition}\label{definition : weak-star topology}
The  weak-* operator topology on $\EL(X,M)$ is defined to be the weakest topology 
for which all operators in $\widehat{B}_{X}$ are continuous when $M$ is equipped with its weak-*
topology.
\end{definition}

Limits in the  weak-* operator topology on $\EL(X,M)$ will be 
denoted ${\mathcal{W}^*}-\lim$. 
The proof  of the following lemma is analogous to the proof of the Banach-Alaoglu theorem.
\begin{lemma}\label{lemma: unit ball is compact in the weak topology}
The unit ball of $\EL(X,M)$ is compact in the  weak-* operator topology.
\end{lemma}

We have the following description of the  weak-* operator topology 
on $\EL(X,M)$.
\begin{proposition}\label{corollary : weak-C topology is pointwise weak* topology}
Let $X$ be a Banach space and let $\{T_{\beta}\}$ be a net in $\EL(X,M)$.
The following conditions are equivalent:
\begin{enumerate}
\renewcommand{\labelenumi}{\normalfont{(\alph{enumi})}}
\item ${\mathcal{W}^*}-\lim_{\beta} T_{\beta}=T$,
\item $w^*-\lim_{\beta} T_{\beta}(x)=T(x)$ in $M$ for every $x\in X$.
\end{enumerate}
\end{proposition}

In the case $Y=\ell_1(G)$ and $Y^*=\C$ we can identify $\EL(X,\C)$ with $\ell_{\infty}(G,X^*)$.
The latter space is naturally the dual of $\ell_1(G,X)$ and can be equipped with the 
corresponding weak-*
topology. The ${\mathcal{W}^*}$-topology and the weak-* topology defined above
 agree on bounded
subsets of $\EL(X,\C)$.

\section{Weak invariant expectations}
In \cite{douglas-nowak} we proved a characterization of exactness in terms of 
invariant expectations; that is, operators whose properties are similar to properties 
of invariant means.
We show that a weak version of such an operator always exists.

\setcounter{aux1}{\value{section}}
\setcounter{aux2}{\value{theorem}}
\setcounter{section}{0}
\setcounter{theorem}{0}

\begin{theorem}\label{theorem : existence of weak invariant expectations}
Let $G$ be a finitely generated group and let $X$ be a bounded Banach $G$-module.
Then there exists a continuous linear map
$$E:\ell_{\infty}(G, \ell_{\infty}(G,X^*))\to \ell_{\infty}(G,X^*),$$
called a weak invariant expectation on $G$ with coefficients in $X^*$,
such that 
\begin{enumerate}
\item\label{item: G-map}$E(g \star f)=g*(E(f))$  for every $g\in G$ and $f\in \ell_{\infty}(G,\ell_{\infty}(G,X^*))$,
\item\label{item: ell_infty map} $E(a f)=a \bullet E(f)$ for every $a \in \ell_{\infty}(G)$ and $f\in \ell_{\infty}(G,\ell_{\infty}(G,X^*))$, and
\item $E={\mathcal{W}^*}-\lim_{\beta}\langle\xi_{\beta},\cdot\rangle_{\ell_{\infty}(G)}$ in 
$\EL(\ell_{\infty}(G,\ell_{\infty}(G,X^*)), \ell_{\infty}(G,X^*))$, where the $\xi_{\beta}\in\ell_{\infty}(G,\ell_{\infty}(G))$ satisfy
\begin{enumerate}
\item every $\xi_{\beta}$ is finitely supported,
\item $\xi_{\beta}\ge 0$ as a function on $G\times G$, and 
\item $\sum_{g\in G}({\xi}_{\beta})_g=1_G$.
\end{enumerate}
\end{enumerate}
\end{theorem}
\setcounter{section}{\value{aux1}}
\setcounter{theorem}{\value{aux2}}

\begin{proof}
Define $E$ by the following formula:
$$(Ef)(g)=f(g,g),$$
where $f \in \ell_{\infty}(G,\ell_{\infty}(G,X^*))$ is viewed as an element of $\ell_{\infty}(G\times G,X^*)$. It is easy to check that (\ref{item: G-map}) and (\ref{item: ell_infty map}) are satisfied.

To prove the last property fix a finite generating set for $G$.
Consider  $\Delta\in \ell_{\infty}(G\times G)$ defined by
$$\Delta(g,h)=\left\lbrace\begin{array}{ll}
1&\text{ if } g=h\\
0&\text{ otherwise}.
\end{array}\right.$$
For a subset $F\subseteq G$ denote by $1_F$ the characteristic function of $F$ and let 
$B(n)\subset G$ denote the ball of radius $n$ centered at the identity element.
Let  $\xi_n:G\to \ell_{\infty}(G)$  be defined by 
$$\xi_{n}=1_{B(n)}\Delta+1_{G\setminus B(n)}\delta_e.$$
The operators $\langle\xi_n,\cdot\rangle_{\C}$ induced by the $\xi_n$ are elements of 
the unit ball of the space $\EL(\ell_{\infty}(G\times G,X^*),\ell_{\infty}(G,X^*))$. 

For any finitely supported $\eta\in \ell_1(G,X)$ we have
\begin{eqnarray*}
\big\langle \langle \xi_n, f\rangle_{\ell_{\infty}(G)},\eta\big\rangle&=
&\sum_{g\in \supp \eta}\left( \left(\langle\xi_n,f\rangle_{\ell_{\infty}(G)}\right)_g\right) (\eta_g)
\end{eqnarray*}
Since the support of $\eta$ is finite, $\supp\eta\subseteq B(n_0)$ for some $n_0$. Then
for all $n\ge n_0$ we have 
$$\left\langle\langle\xi_n,f\rangle_{\ell_{\infty}(G)},\eta \right\rangle=\langle Ef,\eta\rangle.$$
Therefore,
$$w^*-\lim_{n\to \infty}\langle \xi_n,f\rangle_{\ell_{\infty}(G)}=Ef,$$
in $\ell_{\infty}(G,X^*)$,
which shows that $E=\mathcal{W}^*-\lim \langle \xi_n,\cdot\rangle_{\ell_{\infty}(G)}$ and  
proves the claim.
\end{proof}

We remark that a weak invariant expectation with coefficients in $X^*=\RR$ equipped
with a trivial $G$-action is a weak analogue of the invariant expectation considered in
\cite{douglas-nowak}. Indeed, in that case the domain of the weak invariant expectation
is $\ell_{\infty}(G,\ell_{\infty}(G))\simeq \ell_{\infty}(G\times G)$, which is a subspace of the
space $\EL(\ell_u(G),\ell_{\infty}(G))$ of bounded linear maps from the uniform convolution
algebra $\ell_u(G)$ to $\ell_{\infty}(G)$.

\section{Applications}
\subsection*{I. Relative injectivity of Hopf $G$-modules}

Let $X$ be a left Banach $G$-module.

\begin{definition}
A subspace $\mathcal{E}\subseteq \ell_{\infty}(G,X^*)$ is a Hopf $G$-module
if it is both a $G$-submodule and an $\ell_{\infty}(G)$-submodule with respect to 
the actions $*$ and $\bullet$, respectively. 
\end{definition}

Vanishing of bounded cohomology with coefficients in Hopf $G$-modules was studied
in \cite{douglas-nowak}.

The notion of relative injectivity
 is a standard tool in the theory of Hochschild cohomology of Banach
algebras and  bounded cohomology of groups, see  for example
\cite{ivanov,monod-lnm,runde-lectures,sheinberg}, since it
implies the vanishing of cohomology groups in all positive degrees.
The definitions we use are from \cite{monod-lnm}.

A continuous linear map $f:M\to N$ between Banach spaces is admissible if
there is a linear operator $T:N\to M$ such that $\Vert T\Vert\le 1$ and $fTf=f$.
We assume that all $G$-module maps between  bounded Banach $G$-modules are continuous.

\begin{definition}\label{definition : relative injectivity}
A bounded Banach $G$-module $\mathcal{E}$ is relatively injective if for every  injective 
admissible $G$-morphism
$i:M\to N$ and any $G$-morphism $f:M\to \mathcal{E}$ there is a $G$-morphism
$\overline{f}:N\to \mathcal{E}$ such that $\overline{f}\circ i=f$ and $\Vert \overline{f}\Vert\le \Vert f\Vert$.
\end{definition}

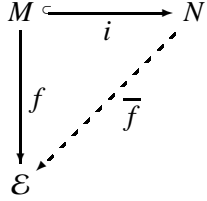
$$\begin{diagram}
M&\rInto_{i}&N\\
\dTo_f&\ldDashto_{\ \ \overline{f}}\\
\mathcal{E}&&
\end{diagram}$$

For a Banach $G$-module $\mathcal{E}$ the module $\ell_{\infty}(G,\mathcal{E})$ is relatively injective \cite{monod-lnm}.
If the injection $\iota:\mathcal{E}\to \ell_{\infty}(G,\mathcal{E})$,
$\iota(x)=x 1_G$, admits a right inverse $E_h$ of norm 1 which commutes with the action of $G$
then the module $\mathcal{E}$ is also relatively injective. Indeed, given the diagram 
$$ 
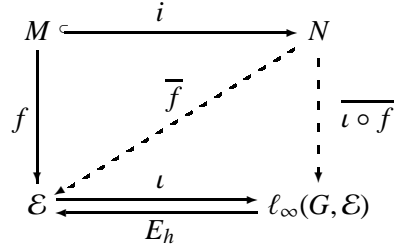
\begin{diagram}
M&\rInto^{i}&N\\
\dTo<f&\ldDashto<{\ \ \overline{f}}&\dDashto>{\ \ \overline{\iota\circ f}}\\
\mathcal{E}&\pile{\rTo^{\ \ \ \ \ \ \ \ \ \ \ \ \ \iota\ \ \ \ \ \ \ \ \ \ \ \ }\\ \lTo_{E_h}}&\ell_{\infty}(G,\mathcal{E})\\
\end{diagram}$$
one verifies that $E_h\circ \left(\overline{\iota\circ f}\right)\circ i=f$
and that $\Vert E_h\circ \overline{\iota\circ f}\Vert=\Vert f\Vert$.

We now use the weak invariant expectation to show that Hopf $G$-modules satisfy the 
conditions of Definition \ref{definition : relative injectivity}.
\begin{theorem}\label{theorem : relative injectivity of Hopf G-modules}
Every weak-* closed Hopf $G$-module is a  relatively injective $G$-module.
\end{theorem}
\begin{proof}
Consider the following diagram:
$$
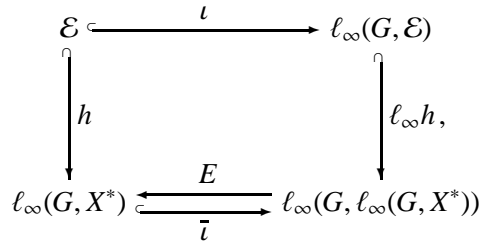
\begin{diagram}
\mathcal{E}&\rInto^{\iota}&\ell_{\infty}(G,\mathcal{E})\\
\dInto_h&&\dInto_{\ell_{\infty}h}\\
\ell_{\infty}(G,X^*)&\pile{\lTo^{\ \ \ \ \ \ \ \ E\ \ \ \ \ \ \ }\\ \rInto_{\overline{\iota}}}&
\ell_{\infty}(G,\ell_{\infty}(G,X^*))
\end{diagram},$$
where $h$ is the natural Hopf inclusion of $\mathcal{E}$ into $\ell_{\infty}(G,X^*)$ 
for some $G$-module $X$, $\ell_{\infty}h$ is induced by applying $h$ coordinate-wise and 
$E$ is a weak invariant expectation. 
Define $E_h:\ell_{\infty}(G,\mathcal{E})\to\ell_{\infty}(G,X^*)$ by 
$$E_h=E\circ \ell_{\infty}h.$$ 
In that case, the explicit formula for $E$ yields
\begin{equation}\label{equation : explicit formula for injectivity}
(E_h\eta)_g=\left(h(\eta_g)\right)_g,
\end{equation}
for every $\eta\in \ell_{\infty}(G,\mathcal{E})$.

By the properties of $E$, for
every $\eta\in \ell_{\infty}(G,\mathcal{E})$ we have
\begin{eqnarray*}
E_h(\eta)&=&w^*-\lim_n\sum_{g\in G} (\xi_n)\eta_g,
\end{eqnarray*}
where the  $\xi_{n}$ are as in Theorem \ref{theorem : existence of weak invariant expectations}.
Since $\eta_g\in \mathcal{E}$, $\xi_n$ is finitely supported and 
$\mathcal{E}$ is closed under the action of $\ell_{\infty}(G)$, we have that
$\sum_{g\in G}(\xi_n)\eta_g$ is an element of $\mathcal{E}$ for every $\beta$. 
Also, $\mathcal{E}$ is weak-* closed and thus the limit belongs to $\mathcal{E}$.

Additionally, for $x\in \mathcal{E}$ it follows from (\ref{equation : explicit formula for injectivity})
that
$$E_h(\iota(x))=E(x1_G)=x,$$
for every $x\in \mathcal{E}$.
The fact that $E_h$ is $G$-equivariant
follows from the properties of $E$ and the fact that $E_h$ is a restriction of $E$ to a 
$G$-invariant subspace. Finally, it is
also easy to verify that $\Vert E_h\Vert=\Vert E\Vert=1$.
\end{proof}

Theorem \ref{theorem : relative injectivity of Hopf G-modules} 
allows one to deduce a vanishing theorem for bounded cohomology with coefficients 
in Hopf $G$-modules.
\begin{theorem}\label{theorem : vanishing of cohomology}
Let $G$ be a finitely generated group. Then the bounded cohomology 
$H^n_b(G,\mathcal{E})$ vanishes for every $n\ge 1$ and every 
weak-* closed Hopf $G$-module $\mathcal{E}$.
\end{theorem}
Theorem \ref{theorem : vanishing of cohomology} follows from \cite[Proposition 7.4.1]{monod-lnm}
and Theorem \ref{theorem : relative injectivity of Hopf G-modules}.

\subsection*{II. A fixed point theorem for actions on $\ell_{\infty}(G,X)$}

The existence of a weak invariant expectation allows one to prove a fixed point theorem for a
group acting on spaces of the type $\ell_{\infty}(G,X)$, where $X$ is a normed space. The fixed
point theorem 
we prove can be viewed as a weak analogue of Day's classical fixed point theorem for 
amenable groups
\cite{day}.

\begin{definition}
A subset $K\subseteq\ell_{\infty}(G,X)$ is called $\ell_{\infty}(G)$-convex if 
given any finite collection of positive elements $ a _1,\dots,  a _n\in\ell_{\infty}(G)$
such that $\sum a _i=1_G$, we have $\sum a _i x_i\in K$ for any $x_1,\dots, x_n\in K$.
\end{definition}

We equip $\ell_{\infty}(G,X)$ with a topology as follows. Every  
$\varphi\in\ell_{\infty}(G,X^*)$ induces a bounded linear operator  $T_{\varphi}:\ell_{\infty}(G,X)\to \ell_{\infty}(G)$ by
the formula
$$T_{\varphi}f (g)= \langle \varphi_g, f_g\rangle.$$
In particular, the inclusion $i:X^*\to \ell_{\infty}(G,X^*)$ as the constant functions 
allows one to interpret each element  of $X^*$ as such an operator.
 
\begin{definition}
Let $V\subseteq X^*$ be a weak-* dense subspace. The ultra-weak topology induced by $V$ on 
$\ell_{\infty}(G,X)$ is the weakest topology with respect to which every 
operator $T_{\varphi}$ induced by
$\varphi\in V$  is continuous, when $\ell_{\infty}(G)$ is equipped with its
natural weak-* topology.
\end{definition}
We will usually omit the reference to $V$.
One important property of the operators induced by elements of $V$ is that they
separate the points of $\ell_{\infty}(G,X)$. This property is crucial in our argument.
Note also that if $X=Y^*$ is itself a dual space, then we can take $V=X\subset X^{**}$. In that case the
ultra-weak topology on $\ell_{\infty}(G,Y^*)$ is precisely the ${\mathcal{W}^*}$-topology on $\ell_{\infty}(G,Y^*)$,
in the sense of the previous sections. 

An action of a group $G$ on a subset $K\subseteq \ell_{\infty}(G,X)$ is said to be $G$-affine if 
 $$g( a  x+ b  y)=(g* a ) gx+(g* b ) gy$$ 
 for $g\in G$, $x,y\in K$ and $ a , b \in \ell_{\infty}(G)$ such that $ a , b \ge 0$ and $ a + b =1_X$.
 Note that such an action is not, in general, inherited from an action on $\ell_{\infty}(G,X)$.

\begin{theorem}\label{theorem : fixed point}
Let $G$ be a finitely generated group, $X$ be a Banach space and $V\subseteq X^*$ 
be a weak-* dense
subspace. Then every $G$-affine action of $G$ on a  bounded,
$\ell_{\infty}(G)$-convex, ultra-weakly compact subset $K\subseteq \ell_{\infty}(G,X)$ has a fixed point.
\end{theorem}

\begin{proof}
We divide the proof into a few lemmas, with the assumptions for each of them being the same.
Fix $\kappa_0\in K$. Let $\Aff(K,\ell_{\infty}(G))$ denote the set of all weak-*
continuous  maps
$T:K\to \ell_{\infty}(G)$ which are $\ell_{\infty}(G)$-convex; that is,
$$T( a  x+ b  y)= a  T(x)+ b  T(y)$$
for $x,y\in K$ and  $ a , b \in \ell_{\infty}(G)$, $ a \ge 0$, $ b  \ge 0$ and 
$ a + b =1_G$.
Observe  that $\ell_{\infty}(G,X^*)\subseteq \Aff(K,\ell_{\infty}(G))$ when restricted to $K$.

The space $\Aff(K,\ell_{\infty}(G))$ plays, rougly speaking, the role of a ``dual space with coefficients in $\ell_{\infty}(G)$''.  
Given $T\in \Aff(K,\ell_{\infty}(G))$ and $g\in G$ define 
$$g\cdot T(x)=g*T(g^{-1}x),$$
for every $x\in K$.
\begin{lemma}
The operation $\cdot$ defines an action of $G$ on $\Aff(K,\ell_{\infty}(G))$.
\end{lemma}
\begin{proof}We only need to show that $g\cdot T$ is $\ell_{\infty}(G)$-convex.
For $ a , b \in \ell_{\infty}(G)$ such that 
$ a \ge 0$, $ b \ge 0$, $ a + b =1_G$ and $x,y\in K$, we have 
\begin{eqnarray*}
(g\cdot T)( a  x+ b  y)&=&g*\left(T(g^{-1}\left( a  x+ b  y\right)\right)\\
&=&g*\left( (g^{-1}* a ) T(g^{-1}x)+(g^{-1}* b ) T(g^{-1}y)\right)\\
&=& a \left(g*T(g^{-1}x)\right)+ b \left(g*T(g^{-1}y)\right)\\
&=& a (g\cdot T)(x)+ b (g\cdot T)(y).
\end{eqnarray*}
\end{proof}

For every $T\in \Aff(K,\ell_{\infty}(G))$ define $f_{[T]}:G\to \ell_{\infty}(G)$ by the formula
$$f_{[T]}(g)=T(g\kappa_0).$$
We have $f_{[T]}\in \ell_{\infty}(G,\C)$.

\begin{lemma}
There exists a point $x\in K$ such that $E(f_{[T]})=T(x)$ for every $T\in \Aff(K,\ell_{\infty}(G))$.
\end{lemma}
\begin{proof}Since $E={\mathcal{W}^*}-\lim_{\beta}\langle\xi_{\beta},\cdot\rangle_{\ell_{\infty}(G)}$ we have
\begin{eqnarray*}
\langle \xi_{\beta}, f_{[T]}\rangle_{\C}&=&\sum_{g\in G}(\xi_{\beta})_g(f_{[T]})_g\\
&=&\sum_{g\in G}(\xi_{\beta})_gT(g\kappa_0)\\
&=&\sum_{g\in G}T((\xi_{\beta})_g g\kappa_0)\\
&=&T\left(\sum_{g\in G}(\xi_{\beta})_g g\kappa_0\right)\\
&=&T\left(x_{\beta}\right),
\end{eqnarray*}
where we used the fact that $T$ is $\ell_{\infty}(G)$-linear and that the $\xi_{\beta}$ 
are finitely supported.
By the ultra-weak 
compactness of $K$ there exists a convergent subnet of the $x_{\beta}$, which we denote
 again by
$x_{\beta}$, and we define 
$x_0=\lim_{\beta}x_{\beta}$. Then for $T\in \Aff(K,\ell_{\infty}(G))$ we have
$$T(x_0)=w^*-\lim_{\beta} T(x_{\beta})=w^*-\lim_{\beta}
\langle f_{[T]},\xi_{\beta}\rangle_{\C}=E(f_{[T]}),$$
by  the ultra-weak continuity of $T$.
\end{proof}

\begin{lemma}
For $g\in G$ we have $f_{[g\cdot T]}=g\star f_{[T]}$.
\end{lemma}
\begin{proof}
For every $h\in G$ we have
\begin{eqnarray*}
\left(f_{[g\cdot T]}\right)_h&=&(g\cdot T)(h\kappa_0)\\
&=&g*\left(T(g^{-1}h\kappa_0)\right)\\
&=&g*\left(\left(f_{[T]}\right)_{g^{-1}h}\right)\\
&=&(g\star f_{[T]})_h.
\end{eqnarray*}
\end{proof}

We now verify that $x_0$ is a fixed point.
For an operator $T\in V\subseteq \Aff(K,\ell_{\infty}(G))$ we obtain
\begin{eqnarray*}
T(gx_0)&=&g*(g^{-1}\cdot T)(x_0)\\
&=&g*E(f_{[g^{-1}\cdot T]})\\
&=&g*E(g^{-1}\star f_{[T]})\\
&=& E(f_{[T]})\\
&=&T(x_0).
\end{eqnarray*}
Since elements of $V$ separate points of $K$,
it follows that $gx_0=x_0$ and $x_0$ is a fixed point, which completes the proof of Theorem 
\ref{theorem : fixed point}.
\end{proof}

We expect that the above fixed point theorem can be generalized to semigroups.

\subsection*{III. Weakly exact Banach algebras}

The above results on bounded cohomology of groups suggest one might define 
a notion of weak exactness for certain Banach algebras. Such algebras have
to be co-algebras in an appropriate sense, so that their duals are Banach algebras in a natural way 
as well. This requirement is a consequence of the fact that we have
used the structure of $\ell_1(G)$ as a Hopf algebra, not only as a Banach algebra.
We will consider only preduals of von Neumann algebras but it is clear that the definition
can be extended to other cases.

Let $M$ be a Hopf-von Neumann algebra and let $A=M_*$ 
denote a predual Banach algebra. 
Let $X$ be a right $A$-module and consider the space $\EL(X,M)$. The algebra 
$M$
is an $A$-bimodule in a natural way, as it is the dual of the $A$-bimodule $A$.
Thus $\EL(X,M)$ is an $A$-bimodule with the following actions:
$$
\begin{array}{ll}
(a\cdot T)(x)&=T(xa),\\
(T\cdot a)(x)&=T(x)a,
\end{array}
$$
for $a\in A$, $T\in \EL(X,M)$ and $x\in X$.
Since $M$ is an algebra, there is the additional structure of an $M$-module on $\EL(X,M)$ 
given by
$$(bT)(x)=b(T(x)),$$
for $b\in M$, $T\in \EL(X,M)$ and $x\in X$.
\begin{definition}
Let $M$ be a Hopf-von Neumann algebra and $A$ its predual Banach algebra.
A submodule of $\EL(X,M)$, which is both an $A$-bimodule  and  an
$M$-module
with respect to the structures described above, is called a Hopf $A$-bimodule.
\end{definition}
Recall that given a Banach algebra $A$ and an $A$-bimodule $\mathcal{E}$ one
can define the Hochschild cohomology groups $\mathcal{H}^*(A,\mathcal{E})$ of 
$A$ with coefficients in $\mathcal{E}$. In particular, the first cohomology group 
$\mathcal{H}^1(A,\mathcal{E})$ is defined as the quotient of the space of all $A$-derivations 
from $A$ into $X$ modulo the inner derivations, see for example \cite{dales-et-al,runde-lectures}.

\begin{definition}\label{definition : exact Banach algebra}
Let $M$ be a Hopf-von Neumann algebra and let $A$ be a predual Banach algebra of $M$.
We define $A$ to be weakly exact if 
$$\mathcal{H}^1(A,\mathcal{E})=0$$
for every $M$-submodule $\mathcal{E}\subseteq \EL(X,M)$, which is closed in the weak-* operator
topology, where $X$ is any left $A$-module. 
\end{definition}

It is natural to ask if dimension shifting preserves the class of Hopf modules over $A$ and,
more importantly, do algebras behave similarly to finitely generated groups:

\begin{question}
Is every Banach algebra $A$ as above weakly exact? 
\end{question}

\end{document}